\documentclass[letterpaper,reqno,12pt]{amsart}
\usepackage[margin=1.2in]{geometry}

\usepackage{amssymb, enumerate}
\usepackage[all]{xy}
\usepackage{hyperref}
\hypersetup{colorlinks=true}

\theoremstyle{plain}
\newtheorem{thm}{Theorem}[section] 
\newtheorem{cor}[thm]{Corollary}
\newtheorem{prop}[thm]{Proposition}

\newtheorem{lem}[thm]{Lemma}
\newtheorem*{mainthm}{Main Theorem}

\theoremstyle{definition} 
\newtheorem{defn}[thm]{Definition}
\newtheorem{setting}[thm]{Setting}
\newtheorem{eg}[thm]{Example} 
\newtheorem{question}[thm]{Question}
\newtheorem*{convention}{Convention}

\theoremstyle{remark}
\newtheorem{rem}[thm]{Remark}

\newtheorem*{acknowledgement}{Acknowledgments}

\def\ge{\geqslant}
\def\le{\leqslant}
\def\phi{\varphi}
\def\epsilon{\varepsilon}

\def\to{\longrightarrow}
\def\mapsto{\longmapsto}

\newcommand{\sO}{\mathcal{O}}

\newcommand{\Q}{\mathbb{Q}} 
\newcommand{\C}{\mathbb{C}} 
 
\newcommand{\Z}{\mathbb{Z}}
\newcommand{\PP}{\mathbb{P}}

\newcommand{\m}{\mathfrak{m}}

\newsavebox{\circlebox}
\savebox{\circlebox}{\fontencoding{OMS}\selectfont\Large\char13}
\newlength{\circleboxwdht}

\def\Spec{\operatorname{Spec}}

\def\Supp{\operatorname{Supp}}

\def\Pic{\operatorname{Pic}}

\title{General hyperplane sections of threefolds\\ in positive characteristic}

\author{Kenta Sato}
\address{Graduate School of Mathematical Sciences, University of Tokyo, 3-8-1 Komaba, Meguro-ku, Tokyo 153-8914, Japan}
\email{ktsato@ms.u-tokyo.ac.jp}

\author{Shunsuke Takagi}
\address{Graduate School of Mathematical Sciences, University of Tokyo, 3-8-1 Komaba, Meguro-ku, Tokyo 153-8914, Japan}
\email{stakagi@ms.u-tokyo.ac.jp}

\thanks{}

\keywords{Bertini theorem, canonical singularities, klt singularities, MJ-canonical singularities, strongly $F$-regular singularities}
\subjclass[2010]{14B05, 14J17, 13A35}

\dedicatory{Dedicated to Professor~Lawrence~Ein on the occasion of his sixty birthday.}

\begin{document}

\begin{abstract}
In this paper, we study the singularities of a general hyperplane section $H$ of a three-dimensional quasi-projective variety $X$ over an algebraically closed field of characteristic $p>0$. 
We prove that if $X$ has only canonical singularities, then $H$ has only rational double points. 
We also prove, under the assumption that $p>5$, that if $X$ has only klt singularities, then so does $H$. 
\end{abstract}

\maketitle
\markboth{K.~SATO and S.~TAKAGI}{GENERAL HYPERPLANE SECTIONS OF THREEFOLDS}

\section{Introduction}

Reid proved in his historic paper \cite{Re} that if a quasi-projective variety over an algebraically closed field of characteristic zero has only canonical singularities, then its general hyperplane section has only canonical singularities, too.  
This can be deduced from the Bertini theorem for base point free linear systems, and the same argument works for other classes of singularities in the minimal model program such as terminal, klt, log canonical singularities. 
Then what if the variety is defined over an algebraically closed field of positive characteristic? 
As a test case, we consider the following question: 

\begin{question}\label{main question}
Let $X$ be a three-dimensional quasi-projective variety over an algebraically closed field $k$ of positive characteristic and $H$ be a general hyperplane section of $X$. 
If $X$ has only terminal singularities, is $H$ smooth? 
If $X$ has only canonical $($resp. klt, log canonical$)$ singularities, does $H$ have only rational double points $($resp. klt singularities, log canonical singularities$)$? 
\end{question}

Since the Bertini theorem for base point free linear systems fails in positive characteristic, 
Reid's argument does not work in this setting.  
Thanks to the recent development in birational geometry, we are able to find a resolution of singularities of $X$ (\cite{CP1}, \cite{CP2}) and to run the minimal model program if $X$ has only klt singularities and the characteristic of $k$ is larger than 5 (\cite{HX}, \cite{Bi}, \cite{BW}), but we do not know in general how to overcome the difficulty arising from the lack of the Bertini theorem. 

The terminal case in Question \ref{main question} is proved affirmatively by the fact that three-dimensional terminal singularities are isolated singularities. 
The other cases are much more subtle.  
For example, Reid proved, as a corollary of the result mentioned above, that all closed points of a complex canonical threefold, except finitely many of them,  have an analytic neighborhood which is nonsingular or isomorphic to the product of a rational double point and $\mathbb{A}^1_{\C}$. 
This does not hold in positive characteristic: Hirokado-Ito-Saito \cite{HIS} and Hirokado \cite{Hi} gave counterexamples in characteristic two and three. 
However, the canonical case in Question \ref{main question} has remained open. 

In this paper, we give an affirmative answer to the canonical and klt cases in Question \ref{main question}, using jet schemes and $F$-singularities, while we need to assume that the characteristic of $k$ is larger than 5 in the klt case. 
In both cases, we look at every codimension two point $x$ of $X$. The normal surface singularity $\Spec \sO_{X, x}$ is canonical or klt but is not defined over an algebraically closed field. 
In the canonical case, we observe, as in the case where the base field is algebraically closed, that $\Spec \sO_{X, x}$ is a hypersurface singularity. 
In the klt case, we deduce from the following theorem, which is a generalization of a result of Hara \cite{Ha}, that $\Spec \sO_{X, x}$ is strongly $F$-regular when the characteristic of $k$ is larger than 5. 

\begin{thm}[Proposition \ref{F-reg klt}, Theorem \ref{klt F-reg}]\label{key thm}
Let $(s \in S)$ be an $F$-finite normal surface singularity $($not necessarily defined over an algebraically closed field$)$ of characteristic $p>5$ and $B$ be an effective $\Q$-divisor on $S$ whose coefficients belong to the standard set $\{1-1/m \, | \, m \in \Z_{\ge 1}\}$. 
Then $(s \in S, B)$ is klt if and only if $(s \in S, B)$ is strongly $F$-regular.  
\end{thm}

 In the canonical case, we see from the above observation that $X$ is MJ-canonical except at finitely many closed points. MJ-canonical singularities are defined in terms of Mather-Jacobian discrepancies instead of the usual discrepancies and can be viewed as a jet scheme theoretic counterpart of canonical singularities. 
Similarly, in the klt case, we see that if the characteristic of $k$ is larger than 5, then $X$ is strongly $F$-regular except at finitely many closed points. Strongly $F$-regular singularities are defined in terms of the Frobenius morphism and can be viewed as an $F$-singularity theoretic counterpart of klt singularities. 
 Applying the Bertini theorem for MJ-canonical singularities due to Ishii-Reguera \cite{IR} to the canonical case and the Bertini theorem for strongly $F$-regular singularities due to Schwede-Zhang \cite{SZ} to the klt case, we obtain the assertion. 
To be precise, we can prove a more general result involving a boundary divisor, which is stated as follows. 

\begin{mainthm}[Proposition \ref{terminal}, Theorems \ref{canonical} and \ref{log terminal}]
Let $X$ be a three-dimensional normal quasi-projective variety over an algebraically closed field $k$ of characteristic $p>0$ and $\Delta$ be an effective $\Q$-divisor on $X$ such that $K_X+\Delta$ is $\Q$-Cartier. Let $H$ be a general hyperplane section of $X$. 
\begin{enumerate}
\item[$(1)$] If $(X, \Delta)$ is terminal, then $(H, \Delta|_H)$ is terminal. 
\item[$(2)$] If $(X, \Delta)$ is canonical, then $(H, \Delta|_H)$ is canonical. 
\item[$(3)$] Suppose in addition that $p>5$ and the coefficients of $\Delta$ belong to the standard set $\{1-1/m \, | \, m \in \Z_{\ge 1}\}$. 
If $(X, \Delta)$ is klt, then $(H, \Delta|_H)$ is klt. 
\end{enumerate}
\end{mainthm}

\begin{small}
\begin{acknowledgement}
The authors are indebted to Shihoko Ishii for helpful discussions, especially on the proof of Theorem \ref{canonical}. 
They are also grateful to Mircea Musta\c{t}\u{a} and Hiromu Tanaka for useful conversations. 
The first author was partially supported by the Program for Leading Graduate Schools, MEXT, Japan. The second author was partially supported by JSPS KAKENHI 26400039. 
\end{acknowledgement}
\end{small}

\begin{convention}
Throughout this paper, all rings are assumed to be commutative and with a unit element and all schemes are assumed to be noetherian and separated. 
A variety over a field $k$ means an integral scheme of finite type over $k$. 
\end{convention}

\section{Preliminaries}

\subsection{$F$-singularities}
In this subsection, we recall the basic notions of $F$-singularities, which we will need in $\S4$. 

Let $X$ be a scheme of prime characteristic $p>0$. 
We say that $X$ is \emph{$F$-finite} if the Frobenius morphism $F:X \to X$ is a finite morphism. 
When $X=\Spec R$ is an affine scheme, $R$ is said to be $F$-finite if $X$ is $F$-finite. 
For example, a field $k$ of positive characteristic $p>0$ is $F$-finite if and only if $[k:k^p]<\infty$. 
It is known by \cite{Ku} that every $F$-finite scheme is locally excellent. 

Suppose in addition that $X$ is a normal integral scheme. 
For an integer $e \ge 1$ and a Weil divisor $D$ on $X$, we consider the following composite map: 
\[
\sO_X \xrightarrow{(F^e)^{\#}} F^e_*\sO_X \xrightarrow{F^e_*\iota} F^e_*\sO_X(D), 
\]
where $F^e:X \to X$ is the $e$-th iteration of Frobenius and $\iota: \sO_X \to \sO_X(D)$ is the natural inclusion. 

Strong $F$-regularity is one of the most basic classes of $F$-singularities. 
\begin{defn}
Let $(R,\m)$ be an $F$-finite normal local ring of positive characteristic $p>0$ and $\Delta$ be an effective $\Q$-divisor on $X=\Spec R$. 
Then the pair $(R, \Delta)$ is said to be \emph{strongly $F$-regular} if for every nonzero element $c \in R$, there exists an integer $e \ge 1$ such that the map 
\[
\sO_X \to F^e_*\sO_X(\lceil (p^e-1) \Delta \rceil +\mathrm{div}_X(c))
\]
splits as an $\sO_X$-module homomorphism. 
When $y$ is a point in an $F$-finite normal integral scheme $Y$ and $B$ is an effective $\Q$-divisor on $Y$, we say that $(y \in Y, B)$ is strongly $F$-regular if $(\sO_{Y, y}, B_y)$ is strongly $F$-regular. 
We say that the normal singularity $(y \in Y)$ is strongly $F$-regular if $(\sO_{Y, y}, 0)$ is strongly $F$-regular. 
\end{defn}

\begin{rem}
We have the following hierarchy of properties of normal singularities: 
\[
\textup{regular}\Longrightarrow \textup{strongly $F$-regular} \Longrightarrow \textup{Cohen-Macaulay}. 
\]
\end{rem}

Global $F$-regularity is defined as a global version of strong $F$-regularity. 
\begin{defn}
Let $X$ be an $F$-finite normal integral scheme and $\Delta$ be an effective $\Q$-divisor on $X$. 
Then $(X, \Delta)$ is said to be \emph{globally $F$-regular} if for every Cartier divisor $D$ on $X$, there exists an integer $e \ge 1$ such that the map 
\[
\sO_X \to F^e_*\sO_X(\lceil (p^e-1) \Delta \rceil +D)
\]
splits as an $\sO_X$-module homomorphism. 
\end{defn}

\begin{rem}
If $X$ is an affine scheme, then $(X, \Delta)$ is globally $F$-regular if and only if $(x \in X, \Delta)$ is strongly $F$-regular for all $x \in X$. 
However, in general, the former is much stronger than the latter. 
For example, a smooth projective variety of general type over an algebraically closed field of positive characteristic is not globally $F$-regular but strongly $F$-regular at all points. 
\end{rem}

\subsection{Singularities in the minimal model program}

In this subsection, we recall the definition of singularities in the minimal model program. Note that we do not assume that the singularities are defined over an algebraically closed field. 
We start with the definition of canonical divisors on normal schemes. 

Suppose that $\pi:Y \to X$ is a separated morphism of finite type between schemes. 
If $X$ has a dualizing complex (see \cite[Chapter V, \S 2]{Hart} for the definition of dualizing complexes), then $\omega^\bullet_Y := \pi^! \omega_X^{\bullet}$ is a dualizing complex on $Y$, where $\pi^!$ is the twisted inverse image functor associated to $\pi$ obtained in \cite[Chapter VII, Corollary 3.4 (a)]{Hart}. 
Since any $F$-finite affine scheme has a dualizing complex by \cite{Ga}, any scheme of finite type over an $F$-finite local ring has a dualizing complex.   

Let $X$ be an excellent integral scheme with a dualizing complex $\omega_X^\bullet$. 
The \emph{canonical sheaf} $\omega_X$ associated to $\omega_X^{\bullet}$ is the coherent $\sO_X$-module defined as the first non-zero cohomology module of $\omega_X^\bullet$.
It is well-known that $\omega_X$ satisfies the Serre's second condition $(S_2)$ (see for example \cite[(1.10)]{Ao}). 
If $X$ is a variety over an algebraically closed field $k$ with structure morphism $f:X \to \Spec k$, then the canonical sheaf associated to $f^!k$ coincides with the classical definition of the canonical sheaf of $X$. 
When $X$ is a normal scheme, a \emph{canonical divisor} $K_X$ on $X$ associated to $\omega_X^\bullet$ is any Weil divisor $K_X$ on $X$ such that $\sO_X(K_X) \cong \omega_X$. 

Suppose that $X$ is a normal scheme with a dualizing complex $\omega_X^{\bullet}$ and fix a canonical divisor $K_X$ on $X$ associated to $\omega_X^\bullet$. 
Given a proper birational morphism $\pi:Y \to X$ from a normal scheme $Y$, we always choose a canonical divisor $K_Y$ on $Y$ which is associated to $\pi^! \omega_X^{\bullet}$ and coincides with $K_X$ outside the exceptional locus of $f$. 

\begin{defn}
(i) We say that $(x \in X)$ is a \emph{normal singularity} if $X=\Spec R$ is an affine scheme with a dualizing complex where $R$ is an excellent normal local ring and if $x$ is the unique closed point of $X$. 

(ii) A proper birational morphism $f: Y \to X$ between schemes is said to be a \emph{resolution of singularities} of $X$, or a \emph{resolution} of $X$ for short, if $Y$ is regular. Suppose that $X$ is a normal integral scheme and $\Delta$ is a $\Q$-divisor on $X$. A resolution $f:Y \to X$ is said to be a \emph{log resolution} if the union of $f^{-1}(\Supp(\Delta))$ and the exceptional locus of $f$ is a simple normal crossing divisor. 
\end{defn}

We are now ready to state the definition of singularities in the minimal model program. 
\begin{defn}\label{mmp}
(i) Suppose that $(x \in X)$ is a normal singularity with a dualizing complex $\omega_X^{\bullet}$ and $\Delta$ is an effective $\Q$-divisor on $X$ such that $K_X+\Delta$ is $\Q$-Cartier. 
Given a proper birational morphism $\pi:Y \to X$ from a normal scheme $Y$, we choose a canonical divisor $K_Y$ on $Y$ associated to $\pi^{!}\omega_X^{\bullet}$ such that  
\[
K_Y+\pi^{-1}_*\Delta=\pi^*(K_X+\Delta)+\sum_i a_i E_i, 
\]
where $\pi^{-1}_*\Delta$ is the strict transform of $\Delta$ by $\pi$, the $a_i$ are rational numbers and the $E_i$ are $\pi$-exceptional prime divisors on $Y$. 
We say that $(x \in X, \Delta)$ is \emph{terminal} (resp.~\emph{canonical}, \emph{klt}, \emph{plt}) if all the $a_i >0$ (resp.~all the $a_i \ge 0$, all the $a_i >-1$ and $\lfloor \Delta \rfloor=0$, all the $a_i >-1$) for every proper birational morphism $\pi:Y \to X$ from a normal scheme $Y$. 
We say that $(x \in X)$ is \emph{terminal} (resp.~\emph{canonical}, \emph{klt}, \emph{plt}) if so is $(x \in X, 0)$. 

(ii) Let $X$ be an excellent integral normal scheme with a dualizing complex and $\Delta$ be an effective $\Q$-divisor on $X$ such that $K_X+\Delta$ is $\Q$-Cartier. 
Then $(X, \Delta)$ is \emph{terminal} (resp.~\emph{canonical}, \emph{klt}, \emph{plt}) if so is the normal singularity $(x \in X=\Spec \sO_{X, x}, \Delta_x)$ for every $x \in X$. 
\end{defn}

\begin{rem}
(i) If $X$ is defined over a field of characteristic zero or $\dim X \le 3$, then it is enough to check the condition in Definition \ref{mmp} only for one $f$, namely, for a log resolution of $(X, \Delta)$. 

(ii) We have the following hierarchy of properties of $\Q$-Gorenstein normal singularities: 
\[
\textup{regular} \Longrightarrow \textup{terminal} \Longrightarrow \textup{canonical} \Longrightarrow \textup{klt}.  
\]
\end{rem}

Strong $F$-regularity implies being klt. 
\begin{prop}[\textup{\cite[Theorem 3.3]{HW}}]\label{F-reg klt}
Let $(x \in X)$ be an $F$-finite normal singularity and $\Delta$ be an effective $\Q$-divisor on $X$ such that $K_X+\Delta$ is $\Q$-Cartier. 
If $(x \in X, \Delta)$ is strongly $F$-regular, then it is klt. 
\end{prop}

\section{Brief review on surface singularities}
In this section, we briefly review the theory of surface singularities. All results are well-known if the singularity is defined over an algebraically closed field. 
Some of the results may be known to experts (even if it is not defined over an algebraically closed field), but we include their proofs here, because we have not been able to find any reference to them. 

\begin{defn}
We say that a scheme $X$ is a \emph{surface} if $X$ is a two-dimensional excellent separated integral scheme with a dualizing complex $\omega_X^\bullet$ and that a normal singularity $(x \in X)$ is a \emph{surface singularity} if $\dim \sO_{X,x}=2$. 
\end{defn}

\begin{rem}\label{remark base scheme}
From now on, we often use the results in \cite{Ko} and \cite{Ta}. 
We remark that the base scheme is assumed to be regular in these references, but this assumption is unnecessary in our framework, because our base scheme is a normal surface singularity and its dualizing complex is unique up to isomorphism. 
\end{rem}

First we recall the definition and basic properties of intersection numbers.
Let $(x \in X)$ be a normal surface singularity and $f :Y \to X$ be a proper birational morphism from a normal surface $Y$ with exceptional curves $E=\bigcup_i E_i$. 
For a Cartier divisor $D$ on $Y$ and a Weil divisor $Z= \sum_{i} a_i E_i$ on $Y$, We define the intersection number $D \cdot Z$ as follows: 
\[D \cdot Z= \sum_{i} a_i \deg_{E_i/k(x)}(\sO_Y(D) |_{E_i}) ,\]
where $\deg_{E_i/k(x)} : \Pic(E_i) \to \Z $ is the degree morphism defined as in \cite[Definition 1.4]{Fu}.

\begin{prop}[\textup{\cite[Proposition 2.5]{Fu}, \cite[Theorem 10.1]{Ko}}]\label{intersection}
With the notation above, the following hold. 
\begin{enumerate}
\item[$(1)$] The intersection pairing 
\[\mathrm{Div}(X) \times \bigoplus_{i} \Z E_i \to \Z, \quad (D, Z) \mapsto D \cdot Z\] is a bilinear map.
\item[$(2)$] If $Z$ is Cartier and $\Supp \, D \subseteq E$, then $D \cdot Z = Z \cdot D$.
\item[$(3)$] If $Y$ is regular, then the intersection matrix $(E_i \cdot E_j)_{i,j}$ is a negative-definite symmetric matrix.
\end{enumerate}
\end{prop}

\begin{defn}
Let $(x \in X)$ be a normal surface singularity with a dualizing complex $\omega_X^{\bullet}$. 
A resolution $f :Y \to X$ with exceptional curves $E=\bigcup_i E_i$ is said to be \emph{minimal} if a canonical sheaf $K_Y$ on $Y$ associated to the dualizing complex $f^! \omega_X^\bullet$ is $f$-nef, that is, $K_Y \cdot E_i \ge 0$ for every $i$. 
A minimal resolution of $X$ always exists by \cite[Theorem 2.25]{Ko}. 
\end{defn}

We also recall the definition of rational singularities. 
A normal surface singularity $(x \in X)$ is said to be a \emph{rational singularity} if $R^1 f_*\sO_Y=0$ for every resolution $f : Y \to X$. 
\begin{prop}\label{klt rational}
Let $(x \in X)$ be a normal surface singularity. 
\begin{enumerate}
\item[$(1)$] $($\cite[Proposition 1.2]{Li}$)$
$(x \in X)$ is a rational singularity if and only if $R^1 f_*\sO_Y=0$ for some resolution $f : Y \to X$.
\item[$(2)$] $($\cite[Proposition 2.28]{Ko}$)$ 
If there exists an effective $\Q$-divisor $\Delta$ on $X$ such that $(x \in X, \Delta)$ is plt, then $(x \in X)$ is a rational singularity. 
\end{enumerate}
\end{prop}

In the rest of this section, we work in the following setting. 

\begin{setting}\label{surface set}
Suppose that $(x \in X)$ is a normal surface singularity with a dualizing complex $\omega_X^\bullet$ and $f: Y \to X$ is a (nontrivial) resolution with exceptional curves $E=\bigcup_i E_i$. 
Let $K_Y$ denote a canonical divisor on $Y$ associated to $f^! \omega_X^\bullet$. 
Let $Z_f$ denote the fundamental cycle of $f$, that is, a unique minimal nonzero effective divisor $Z$ on $Y$ such that $\Supp \, Z \subseteq E$ and $-Z$ is $f$-nef $($such a divisor always exists by \cite[Definition and Claim 10.3.6]{Ko}$)$. 
\end{setting}

Let $D$ be a divisor on $Y$ and $Z$ be a nonzero effective divisor on $Y$ such that $\Supp \,Z \subseteq E$.
The Euler characteristic $\chi(Z, D)$ is defined by 
\[\chi(Z, \sO_Y(D)|_Z) = \dim_{k(x)} H^0(Z, \sO_Y(D)|_Z) - \dim_{k(x)} H^1(Z, \sO_Y(D)|_Z).\]

The following is the Riemann-Roch theorem for curves embedded in a regular surface.

\begin{prop}\label{RR}
For each divisor $D$ on $Y$ and each nonzero effective divisor $Z=\sum_i a_i E_i$ on $Y$ such that $\Supp \,Z \subseteq E$, we have
\[
\chi(Z, D)=\frac{(2D-K_Y-Z) \cdot Z}{2}=\chi(Z, 0) + D \cdot Z.
\]
\end{prop}

\begin{proof}
We will prove the assertion by induction on $a:= \sum_{i} a_i$.
If $a=1$, then $Z=E_i$ for some $i$ and the assertion immediately follows from \cite[Theorem 2.5 (1) and Theorem 2.7]{Ta} (see also Remark \ref{remark base scheme}). 
Suppose that $a \ge 2$. 
Then there exists $i$ such that $a_i \ge 1$, so set $Z' :=Z - E_i$ and consider the exact sequence
\[
0 \to \sO_{E_i}(-Z') \to \sO_Z \to \sO_{Z'} \to 0.
\]
By the additivity of the Euler characteristic and applying the induction hypothesis to $Z'$ and $E_i$, one has 
\begin{align*}
\chi (Z, D) &= \chi (Z', D)+ \chi (E_i, D-Z')\\
&= \frac{(2D-K_Y-Z') \cdot Z'}{2} + \frac{(2(D-Z')-K_Y-E_i) \cdot E_i}{2}\\
&=\frac{(2D-K_Y-Z) \cdot Z}{2}.
\end{align*}
\end{proof}

\begin{prop}[\cite{Ar}]\label{rat fund}
Suppose that $(x \in X)$ is a rational singularity. 
Let $\m_x$ denotes the maximal ideal on $X$ corresponding to $x$ and $e(x \in X)$ be the Hilbert-Samuel multiplicity of $\sO_{X, x}$. 
Then the following hold for each integer $n \ge 0$. 
\begin{enumerate}
\item[$(1)$] $\m_x^n = H^0(Y, \sO_Y(-n Z_f))$.
\item[$(2)$] $\dim_{k(x)} \m^n/\m^{n+1}= -n Z_f^2 + 1$.
\item[$(3)$] $e(x \in X)=-Z_f^2$.
\end{enumerate}
\end{prop}

\begin{proof}
First note that $\m_x^n \sO_Y=\sO_Y(-n Z_f)$, which follows from essentially the same argument as for \cite[Theorem 4]{Ar}, where $X$ is assumed to be defined over an algebraically closed field. 
Since in a rational surface singularity, the product of integrally closed ideals is again integrally closed (see \cite[Theorem 7.2]{Li}), $\m_x^n$ is integrally closed. 
Therefore, 
\[\m_x^n=H^0(Y, \m^n_x \sO_Y)= H^0(Y, \sO_Y(-n Z_f)).\] 

For (2), we consider the exact sequence 
\[0 \to \sO_Y(-(\ell+1)Z_f) \to \sO_Y(-\ell Z_f) \to \sO_{Z_f}(- \ell Z_f) \to 0\]
for each integer $\ell \ge 0$. 
Since $H^1(Y, \sO_Y(-\ell Z_f))=H^1(Y, \sO_Y(-(\ell+1) Z_f))=0$ by \cite[Proposition 10.9 (1)]{Ko}, 
\begin{align*}
\chi (Z_f, -\ell Z_f)&=\dim_{k(x)} H^0(Z_f, \sO_{Z_f}(-\ell Z_f))\\
&=\dim_{k(x)} H^0(Y, \sO_Y(-\ell Z_f))/H^0(Y, \sO_Y(-(\ell+1) Z_f))\\
&=\dim_{k(x)} \m_x^{\ell}/\m_x^{\ell+1}, 
\end{align*}
where the last equality follows from (1). 
In particular, $\chi(Z_f , 0)=1$.
Thus, wee see from Proposition \ref{RR} that 
\[
\dim_{k(x)} \m_x^{n}/\m_x^{n+1}=\chi (Z_f, -n Z_f)=1-n Z_f^2.
\]

Finally, (3) immediately follows from (2), because 
\[
\dim_{k(x)} \sO_{X,x}/\m_x^{n}=\sum_{i=0}^{n-1} \dim_{k(x)} \m_x^{i}/\m_x^{i+1}=\frac{-Z_f^2}{2}n^2+O(n). 
\] 
\end{proof}

We will use the following result to prove the canonical case of the main theorem in $\S3$. 
\begin{prop}\label{2-dim canonical}
A normal surface singularity $(x \in X)$ is canonical if and only if $(x \in X)$ is either a regular point or a rational double point. In particular, a canonical surface singularity is a hypersurface singularity. 
\end{prop}

\begin{proof}
We may assume by Proposition \ref{klt rational} that $(x \in X)$ is a rational singularity. 
Since the assertion is obvious when $X$ is regular, we can also assume that $(x \in X)$ is not regular.

Suppose that $f:Y \to X$ is a minimal resolution, and let $\Delta=f^*K_X-K_Y$. 
Then $\Delta$ is an anti-$f$-nef $\Q$-divisor on $Y$, so $\Delta \ge 0$ by \cite[Claim 10.3.5]{Ko}. 
On the other hand, it follows from Proposition \ref{RR} that 
\[
2=2 \chi(Z_f, 0)=-(K_Y+Z_f) \cdot Z_f=-K_Y \cdot Z_f-Z_f^2=\Delta \cdot Z_f - Z_f^2, 
\]
which implies that $\Delta \cdot Z_f=Z_f^2+2$. 

If $(x \in X)$ is a canonical singularity, then we have $\Delta=0$ and hence $Z_f^2=-2$.
We see by Proposition \ref{rat fund} (3) that $e(x \in X)=2$, that is, $(x \in X)$ is a rational double point. 
Using Proposition \ref{rat fund} (2), (3), we can verify that every rational double point is a hypersurface singularity.

Conversely, if $(x \in X)$ is a rational double point, then $Z_f^2=-2$ by Proposition \ref{rat fund} (3) and therefore $\Delta \cdot Z_f=0$.
Since $-\Delta$ is $f$-nef and $\mathrm{Supp} \, Z_f=E$, the intersection number $\Delta \cdot E_i$ has to be zero for all $i$.
This means by \cite[Claim 10.3.5]{Ko} that $\Delta=0$. 
It then follows from the fact that $K_Y+\Delta=f^*K_X$ and $(Y, \Delta)=(Y, 0)$ is canonical that $(x \in X)$ is a canonical singularity. 
\end{proof}

\section{Terminal and canonical singularities}

The terminal case of the main theorem is an immediate consequence of a Bertini theorem for Hilbert-Samuel multiplicity. 
\begin{prop}\label{terminal}
Let $X$ be a three-dimensional normal quasi-projective variety over an algebraically closed field of characteristic $p>0$ and $\Delta$ be an effective $\Q$-divisor on $X$ such that $K_X+\Delta$ is $\Q$-Cartier.  
Suppose that the pair $(X, \Delta)$ is terminal. 
Then $(H, \Delta|_H)$ is terminal for a general hyperplane section $H$ of $X$.  
In particular, $H$ is smooth. 
\end{prop}

\begin{proof}
Let $U$ be the locus of the points $x \in X$ such that $X$ is regular at $x$ and $\mathrm{mult}_x(\Delta) < 1$. 
Since the regular locus of $X$ is open and Hilbert-Samuel multiplicity is upper semicontinuous, 
$U$ is an open subset of $X$. 
Then every codimension two point $x \in X$ lies in $U$ by \cite[Theorem 2.29 (1)]{Ko}, because  
$(\Spec \sO_{X,x}, \Delta_x)$ is a two-dimensional terminal pair. 
Hence, $X \setminus U$ consists of only finitely many closed points and a general hyperplane section $H$ of $X$ is contained in $U$. 
It follows from a Bertini theorem for Hilbert-Samuel multiplicity \cite[Proposition 4.5]{dFEM}, which holds in arbitrary characteristic, that $H$ is smooth and $\mathrm{mult}_x(\Delta|_{H})=\mathrm{mult}_x(\Delta) <1$ for all $x \in \mathrm{Supp}\, \Delta|_{H}$. 
Thus, applying \cite[Theorem 2.29 (1)]{Ko} again, we see that $(H, \Delta|_{H})$ is terminal. 
\end{proof}

The canonical case of the main theorem can be deduced from a Bertini theorem for MJ-canonical singularities. 
The proof was inspired by a discussion with Shihoko Ishii, whom we thank. 

\begin{thm}\label{canonical}
Let $X$ be a three-dimensional normal quasi-projective variety over an algebraically closed field $k$ of characteristic $p>0$ and $\Delta$ be an effective $\Q$-divisor on $X$ such that $K_X+\Delta$ is $\Q$-Cartier.  
Suppose that the pair $(X, \Delta)$ is canonical. 
Then $(H, \Delta|_H)$ is canonical for a general hyperplane section $H$ of $X$. In particular, $H$ has only rational double points. 
\end{thm}
\begin{proof}
First we will show the case where $\Delta=0$. 
Since $X$ has only canonical singularities, $\Spec \sO_{X,x}$ is a two-dimensional scheme with only canonical singularities for every codimension two point $x \in X$. 
It then follows from Proposition \ref{2-dim canonical} that $\sO_{X,x}$ is a hypersurface singularity for every codimension two point $x \in X$. 
Since the l.c.i.~locus of $X$ is open by \cite{GM}, $X$ is a l.c.i.~except at finitely many closed points. 
By the fact that l.c.i.~canonical singularities are MJ-canonical (see for example \cite[Remark 2.7 (iv)]{IR}), $X$ has only MJ-canonical singularities except at finitely many closed points. 
Thus, applying a Bertini theorem for MJ-canonical singularities \cite[Corollary 4.11]{IR}, we see that $H$ has only rational double points. 

Next we consider the case where $\Delta \ne 0$. 
Let $U_1$ be the locus of the points $x \in X$ such that $X$ is regular at $x$ and $\mathrm{mult}_x(\Delta) \le 1$, and $U_2=X \setminus \mathrm{Supp} \,\Delta$. 
Note that $U_1$ and $U_2$ are both open subsets of $X$ (the openness of $U_1$ follows from the openness of the regular locus of $X$ and the upper-semicontinuity of Hilbert-Samuel multiplicity). 
Then every codimension two point $x \in X$ lies in $U_1 \cup U_2$ by \cite[Theorem 2.29 (2)]{Ko}, because $(\Spec \sO_{X, x}, \Delta_x)$ is a two-dimensional canonical pair. 
Hence, we may assume that $X=U_1 \cup U_2$. 
We have already seen that the assertion holds when $\Delta=0$, so $(H \cap U_2, \Delta|_{H \cap U_2})=(H \cap U_2, 0)$ is canonical for a general hyperplane section $H$ of $X$. 
On the other hand, it follows from a Bertini theorem for Hilbert-Samuel multiplicity \cite[Proposition 4.5]{dFEM} that $H \cap U_1 \subsetneq U_1$ is smooth and $\mathrm{mult}_x(\Delta|_{H \cap U_1})=\mathrm{mult}_x(\Delta) \le 1$ for all $x \in \mathrm{Supp}\, \Delta|_{H \cap U_1}$. 
Thus, applying \cite[Theorem 2.29 (2)]{Ko} again, we see that $(H \cap U_1, \Delta|_{H \cap U_1})$ is canonical. 
\end{proof}

\begin{cor}\label{product}
Let $X$ be a three-dimensional quasi-projective variety over an algebraically closed field $k$ of characteristic $p \ge 5$ with only canonical singularities. 
Then there exists a zero-dimensional closed subscheme $Z \subset X$ with the following property: for every closed point $x \in X \setminus Z$, the maximal-adic completion $\widehat{\sO_{X,x}}$ of the local ring of $X$ at $x$ is either regular or isomorphic to the formal tensor product $k[[u,v,w]]/(f) \widehat{\otimes}_k k[[t]]$ of a rational double point $k[[u,v,w]]/(f)$ and the one-dimensional formal power series ring $k[[t]]$. 
\end{cor}
\begin{proof}
An immediate application of Theorem \ref{canonical} to \cite[Theorem 3]{HIS} yields the desired result. 
\end{proof}

\begin{rem}
When the characteristic $p$ is less than $5$, there exist counterexamples to Corollary \ref{product} due to Hirokado-Ito-Saito \cite{HIS} and Hirokado \cite{Hi}. 
In particular, in characteristic three, an exhaustive list of examples is given in \cite[Theorem 3]{HIS}. 
\end{rem}

\section{Klt singularities}

In this section, we will prove the klt case of the main theorem. 
Since we often use the results of \cite{Ko} and \cite{Ta} in this section, the reader is referred to Remark \ref{remark base scheme}. 

First we prepare some definitions and lemmas for Theorem \ref{key thm}. 
\begin{defn}
Let $I \subseteq [0,1]$ be a subset.
We define the subset $D(I) \subseteq [0,1]$ by
\[ D(I) := \left\{ \frac{m-1+ \sum_{j=1}^n i_j}{m} \; \Bigg| \;  m \in \Z_{\ge 1}, n \in \Z_{\ge 0}, i_j \in I \textup{ for } j=1, \dots , n \right\} \cap [0,1] .\]
In particular, $D(\emptyset)=\{1-1/m \, | \, m \in \Z_{\ge 1}\}$ is the set of standard coefficients. 
\end{defn}

\begin{lem}[\textup{cf. \cite[Lemma 4.3]{MP}}]\label{adj coeff}
Let $X$ be a normal surface and $\Delta$ be an effective $\Q$-divisor on $X$ such that $(X, \Delta)$ is plt. 
Let $C$ be a regular curve which is an irreducible component of $\lfloor \Delta \rfloor$ and $\mathrm{Diff}_{C}(B)$ be the different of $B:=\Delta-C$ on $C$ $($see \cite[Definition 2.34]{Ko} for the definition of the $\Q$-divisor $\mathrm{Diff}_{C}(B))$. 
If the coefficients of $B$ belong to a subset $I \subset [0,1] \cap \Q$, then the coefficients of $\mathrm{Diff}_{C}(B)$ belong to $D(I)$.
\end{lem}

\begin{proof}
We may assume that $(x \in X)$ is a surface singularity, that is, $X=\Spec R$ for a two-dimensional normal local ring. 
Let $f: Y \to X$ be a minimal log resolution of $(X, C)$ with exceptional curves $E=\bigcup_i E_i$ defined as in \cite[Definition 2.25 (b)]{Ko} and $m$ be the determinant of the negative of the intersection matrix of the $E_i$.
Then by \cite[(3.36.1)]{Ko}, 
\[ \mathrm{Diff}_{C}(B) = \frac{m-1}{m}  [x] + B|_{C}. \]
Let $B=\sum_j b_j B_j$ be the irreducible decomposition of $B$. 
Since $m D$ is Cartier for every Weil divisor $D$ on $X$ by \cite[Proposition 10.9 (3)]{Ko}, we see that $\mathrm{ord}_{[x]} (mB_j)|_C$ is an integer. 
Thus, 
\[
\mathrm{ord}_{[x]}\mathrm{Diff}_{C}(B)=\frac{m-1}{m}+\sum_j \frac{b_j}{m} \mathrm{ord}_{[x]} (mB_j)|_C \in D(I), 
\]
which completes the proof of the proposition. 
\end{proof}

A \textit{log Fano} pair $(X, \Delta)$ is a pair of a normal projective variety over a field $k$ and an effective $\Q$-divisor on $X$ such that $-(K_X+\Delta)$ is ample and $(X, \Delta)$ is klt. 

\begin{prop}[\textup{cf.~\cite[Corollary 4.1]{CGS}, \cite[Theorem 4.2]{Wa}}]\label{CGS P1}
Let $I \subseteq [0,1] \cap \Q$ be a finite set.
There exists a positive constant $p_0(I)$ depending only on $I$ with the following property:
if $k$ is an $F$-finite field of characteristic $p>p_0(I)$ and $(\PP^1_k, B= \sum_{i=1}^m a_i P_i)$ is a log Fano pair such that every coefficient $a_i$ belongs to $D(I)$ and every point $P_i$ is a $k$-rational point of $\PP^1_k$, then $(\PP^1_k, B)$ is globally $F$-regular.
Moreover, when $I= \emptyset$, we may take the constant $p_0(I)$ to be 5. 
\end{prop}

\begin{proof}
The proof is essentially the same as in the case where $k$ is algebraically closed. 
When $k$ is algebraically closed, the first assertion is nothing but \cite[Theorem 4.1]{CGS} and the second one follows from \cite[Theorem 4.2]{Wa}.   
\end{proof}

\begin{defn}
A \emph{curve} $X$ over a field $k$ is a one-dimensional integral scheme of finite type over $k$. 
The \emph{arithmetic genus} $g(X)$ of a complete curve $X$ over $k$ is defined to be $\dim_k H^1(X, \sO_X)/ \dim_k H^0(X, \sO_X)$.  
We remark that $g(X)$ is independent of the choice of the base field $k$.
\end{defn}

We extend Proposition \ref{CGS P1} to the case of an arbitrary curve of genus zero. 
\begin{prop}\label{CGS genus0}
Let $I \subseteq [0,1] \cap \Q$ be a finite set.
There exists a positive constant $p_1(I)$ depending only on $I$ with the following property:
if $k$ is an $F$-finite field of characteristic $p>p_1(I)$ and $(C, B)$ is a log Fano pair such that $C$ is a complete curve over $k$ 
and the coefficients of $B$ belong to $D(I)$, then $(C, B)$ is globally $F$-regular. 
Moreover, when $I=\emptyset$, we may take $p_1(I)$ to be 5. 
\end{prop}

\begin{proof}
Let $p_0(I)$ be the constant as in Proposition \ref{CGS P1}.
Set $a(I)= \min\{a  \in D(I) \, | \, a \ne 0 \}$ and $p_1(I)= \max \{ p_0, \lceil 2/a(I) \rceil \}$. 
If $I=\emptyset$, then $a(I)=1/2$, so we can take $p_1(\emptyset)=5$ by Proposition \ref{CGS P1}. 
We say that a pair $(C, B)$ satisfies the condition $(\star)$ if $C$ is a complete curve over an $F$-finite field of characteristic $p>p_1(I)$ 
and $(C, B)$ is a log Fano pair such that the coefficients of $B$ belong to $D(I)$. 
Note that if $(C, B)$ satisfies $(\star)$, then $C$ is a complete regular curve with $g(C)=0$ by \cite[Corollary 2.6]{Ta}.

Let $(C, B)$ be a log Fano pair over a field $k$ satisfying the condition $(\star)$. 
First, we will show that we can reduce to the case where $C \cong \PP^1_k$. 
Replacing $k$ by its finite extension if necessary, we may assume that $k=H^0(C, \sO_C)$.
Suppose that $C \not\cong \PP^1_k$. 
Then $C$ is isomorphic to a conic in $\PP^2_k$ by \cite[Lemma 10.6 (3)]{Ko}, and hence there exists a closed point $P \in C$ whose residue field $k(P)$ is a quadratic extension of $k$. 
Let $l := k(P)$ and consider the pullback $f : C_l:= C \times_{\Spec k} \Spec l \to C$ of $C$.  
Note that $C_l$ is connected, because $[l : k]=2$ and $l$ is not contained in the function field $K(C)$ of $C$, so $K(C) \otimes_k l$ is a domain. 
Since $p > p_1 >2= [l :k]$, the pullback $f$ is a surjective \'{e}tale morphism. 
It follows from \cite[Lemma 10.6 (3)]{Ko} again, together with the fact that a regular conic with an $l$-rational point is isomorphic to $\PP^1_l$, that $C_l \cong \PP^1_l$. 
By the fact that being klt is preserved under surjective \'{e}tale morphisms (\cite[Proposition 2.15]{Ko}), the pair $(C_l, B_l:=f^*B)$ satisfies the condition $(\star)$.  
On the other hand, a pair is globally $F$-regular if and only if its affine cone is strongly $F$-regular (\cite[Proposition 5.3]{SS}), and strong $F$-regularity is preserved under finite \'{e}tale morphisms of degree not divisible by $p$ (\cite[Theorem 3.3]{HT}, \cite[Corollary 6.31 and Proposition 7.4]{ST}). 
Therefore, $(C, B)$ is globally $F$-regular if and only if so is $(C_l, B_l)$. 
Thus, replacing $k$ by $l$ and $C$ by $C_l$, we may assume that $C = \PP^1_k$.

Next, we will show that we can reduce to the case where every point in $\Supp \, B$ is a $k$-rational point. 
Let $B= \sum_{i} a_i P_i$ be the irreducible decomposition of $B$.
Since $(C=\PP^1_k,B)$ is log Fano, we have $\sum_{i} a_i [k(P_i) : k]=\deg_{\PP^1_k/k}(B) < 2$.
In particular, $[k(P_i) : k] < 2/ a(I) \le p_1(I)<p$ for every $i$.
Hence, there exists a Galois extension $K/k$ such that $k(P_i) \subseteq K$ for every $i$ and $p$ does not divide $[K:k]$.
Since the condition $(\star)$ and global $F$-regularity are preserved under finite Galois base field extensions of degree not divisible by $p$ as we have seen above, 
replacing $k$ by $K$ and $(\PP^1_k, B)$ by $(\PP^1_K, B_K)$ if necessary, we may assume that every point in $\Supp \, B$ is a $k$-rational point. 
The assertion is now an immediate consequence of Proposition \ref{CGS P1}.
\end{proof}

\begin{lem}[\textup{cf.~\cite[Proposition 2.13]{CGS}}]\label{Kollar comp}
Let $(x \in X)$ be a normal surface singularity with a dualizing complex $\omega_X^{\bullet}$ and $B$ be an effective $\Q$-divisor on $X$ such that $K_X+B$ is $\Q$-Cartier. 
If $(x \in X, B)$ is klt, then there exists a proper birational morphism $f : Y \to X$ from a normal surface $Y$ such that 
\begin{enumerate}
\item[$(1)$] the exceptional locus $C$ of $f$ is a complete regular curve with $g(C)=0$,
\item[$(2)$] $(Y, B_Y+C)$ is plt, where $B_Y = f^{-1}_*B$ is the strict transform of $B$ by $f$, and
\item[$(3)$] $-(K_Y + B_Y +C)$ is $f$-ample, where $K_Y$ is a canonical divisor on $Y$ associated to the dualizing complex $f^! \omega_X^\bullet$. 
\end{enumerate}
\end{lem}

\begin{proof}
We can apply essentially the same argument as in the case where $X$ is defined over an algebraically closed field \cite[Proposition 2.13]{CGS}, using the minimal model program for excellent surfaces \cite[Theorem 1.1]{Ta} instead of that for surfaces over an algebraically closed field. 
Then (2) and (3) are easily verified. 
Hence, it is enough to check (1). 
Since $(Y, B_Y+C)$ is plt, the complete curve $C$ has to be regular by \cite[Theorem 2.31]{Ko}. 
The pair $(Y, B_Y+C)$ being plt also implies by Proposition \ref{klt rational} that $Y$ has only rational singularities, from which it follows that $g(C)=0$ (see \cite[Lemma 10.8 (3)]{Ko}). 
\end{proof}

Theorem \ref{key thm} is a consequence of Proposition \ref{F-reg klt} and the following theorem, which gives a generalization of \cite[Theorem 1.1]{CGS} (see also \cite{Ha}) to the case where the base field is not necessarily algebraically closed. 
\begin{thm}[\textup{cf.~\cite{Ha}, \cite[Theorem 1.1]{CGS}}]\label{klt F-reg}
Let $I \subseteq [0,1] \cap \Q$ be a finite set.
There exists a positive constant $p_1(I)$ depending only on $I$ with the following property:
if $(x \in X)$ be an $F$-finite normal surface singularity of characteristic $p>p_1(I)$ and $(x \in X, B)$ is a klt pair such that the coefficients of $B$ belong to $D(I)$,
then $(x \in X, B)$ is strongly $F$-regular.
Moreover, when $I= \emptyset$, we can take $p_1(I)$ to be 5.
\end{thm}

\begin{proof}
We employ the same strategy as the proof of \cite[Theorem 1.1]{CGS}.
Let $p_1(I)$ be the constant as in Proposition \ref{CGS genus0}. 
Suppose that $(x \in X, B)$ is a klt pair such that $(x \in X)$ is an $F$-finite normal surface singularity of characteristic $p>p_1(I)$ and the coefficients of $B$ belong to $D(I)$. 
Take a proper birational morphism $f: Y \to X$ with exceptional prime divisor $C$ as in Lemma \ref{Kollar comp}, and let $B_Y=f^{-1}_* B$ be the strict transform of $B$ by $f$ and $B_C=\mathrm{Diff}_C(B_Y)$ be the different of $B_Y$ on $C$. 
Note that by Lemma \ref{adj coeff} and \cite[Lemma 4.4]{MP}, the coefficients of $B_C$ belong to $D(D(I))=D(I) \cup \{1\}$.
Since $(C, B_C)$ is log Fano by adjunction (\cite[Theorem 3.36 and Lemma 4.8]{Ko}), 
the coefficients of $B_C$ are less than 1, and then $(C, B_C)$ is globally $F$-regular by Proposition \ref{CGS genus0}. 
Applying the same argument as in the case where $X$ is defined over an algebraically closed field (\cite[Proposition 2.11]{CGS} and \cite[Lemma 2.12]{HX}), we can conclude that $(X, B)$ is strongly $F$-regular. 
\end{proof}

\begin{eg}[\textup{cf.~\cite[2.26]{Ko}}]
Let $k$ be a non-algebraically closed field of characteristic $p>0$ with an element $a \in k$ that is not a  cubic power in $k$.  
Let $(\mathbf{0} \in X)$ be the origin of the hypersurface $(x^2=y^3-az^3+y^4+z^4) \subset \mathbb{A}_k^3$, which is a klt surface singularity. 
Its minimal resolution has two exceptional curves $C_1, C_2$ and their dual graph is 
\vspace*{-0.25em}
\[
\xygraph{
[]*=[o]++[Fo]{{2}} ([]!{+(0,+.25)} {{}^1})
- @3{-}[r]*=[o]++[Fo]{{2}}([]!{+(0,+.25)} {{}^3}) 
} \, ,
\]
where the numbers above the circles are the $\dim_k H^0(C_i, \sO_{C_i})$ and the numbers inside the circles are the $-C_i^2/\dim_k H^0(C_i, \sO_{C_i})$. 
Note that this graph does not appear in the list of the dual graphs of exceptional curves for the minimal resolutions of klt surface singularities over an algebraically closed field (see \cite{Wk}).  
On the other hand, $(\mathbf{0} \in X)$ is strongly $F$-regular if and only if $p >3$. 
\end{eg}

We are now ready to prove the klt case of the main theorem. 
\begin{thm}\label{log terminal}
Let $X$ be a three-dimensional normal quasi-projective variety over an algebraically closed field of characteristic $p > 5$ and $\Delta$ be an effective $\Q$-divisor on $X$ whose coefficients belong to the standard set $\{1-1/m \, | \, m \in \Z_{\ge 1}\}$. 
Suppose that $K_X+\Delta$ is $\Q$-Cartier and $(X, \Delta)$ is klt. 
Then $(H, \Delta|_H)$ is klt for a general hyperplane section $H$ of $X$. 
\end{thm}

\begin{proof}
Let $U$ be the locus of the points $x \in X$ such that $(X, \Delta)$ is strongly $F$-regular. 
Note that $U$ is an open subset of $X$ (the $\Delta=0$ case was proved in \cite[Theorem 3.3]{HH} and the general case follows from a very similar argument). 
Every codimension two point $x \in X$ lies in $U$, because $(\Spec \sO_{X, x}, \Delta_x)$ is strongly $F$-regular by Theorem \ref{klt F-reg}. 
Hence, $X \setminus U$ consists of only finitely many closed points and a general hyperplane section $H$ of $X$ is contained in $U$. 
It then follows from a Bertini theorem for strongly $F$-regular pairs \cite[Corollary 6.6]{SZ} that $(H, \Delta|_H)$ is strongly $F$-regular. 
Since strongly $F$-regular pairs are klt by Proposition \ref{F-reg klt}, we obtain the assertion. 
\end{proof}



\begin{thebibliography}{99}

\bibitem{Ao} Y.~Aoyama, Some basic results on canonical modules, J. Math. Kyoto. Univ. \textbf{23} (1983), 85--94. 

\bibitem{Ar} M.~Artin, On isolated rational singularities of surfaces, Amer. J. Math. \textbf{88} (1966), 129--136.

\bibitem{Bi}
C.~Birkar, Existence of flips and minimal models for 3-folds in char $p$, Ann. Sci. \'{E}c. Norm. Sup\'{e}r. (4) \textbf{49} (2016), 169--212.

\bibitem{BW}
C.~Birkar and J.~Waldron, Existence of Mori fibre spaces for 3-folds in char $p$, arXiv:1410.4511, preprint (2014). 

\bibitem{CGS}
P.~Cascini, Y.~Gongyo and K.~Schwede, Uniform bounds for strongly $F$-regular surfaces, Trans. Amer. Math. Soc. \textbf{368} (2016), 5547--5563.

\bibitem{CP1}
V.~Cossart, O.~Piltant, Resolution of singularities of threefolds in positive characteristic. I, J. Algebra \textbf{320} (2008), 1051--1082.

\bibitem{CP2}
V.~Cossart, O.~Piltant, Resolution of singularities of threefolds in positive characteristic. II, J. Algebra \textbf{321} (2009), 1836--1976.

\bibitem{dFEM} T.~de~Fernex, L.~Ein, M.~Musta\c{t}\u{a}, Bounds for log canonical thresholds with applications to birational geometry, Math. Res. Lett. \textbf{10} (2003), 219--236.

\bibitem{Fu} W.~Fulton, \textit{Intersection theory}, second edition, Ergebnisse der Mathematik und ihrer Grenzgebiete, 3. Folge, Band 2, Springer, Berlin, 1998. 

\bibitem{Ga} O.~Gabber, Notes on some $t$-structures, \textit{Geometric aspects of Dwork theory. Vol. I, II}, pp. 711--734, Walter de Gruyter GmbH $\&$ Co. KG, Berlin, 2004.

\bibitem{GM} S.~Greco, M.~G.~Marinari, Nagata's criterion and openness of loci 
for Gorenstein and complete intersection, Math. Z. \textbf{160} (1978), 207--216. 

\bibitem{HX} C.~Hacon and C.~Xu, On the three dimensional minimal model program in positive
characteristic, J. Amer. Math. Soc. \textbf{28} (2015), 711--744. 

\bibitem{Ha} N.~Hara, Classification of two-dimensional $F$-regular and $F$-pure singularities, Adv.~Math. \textbf{133} (1998), 33--53. 

\bibitem{HT} N.~Hara and S.~Takagi, On a generalization of test ideals, Nagoya Math. J. \textbf{175} (2004), 59--74.

\bibitem{HW} N.~Hara and K.-i.~Watanabe, $F$-regular and $F$-pure rings vs. log terminal and log canonical singularities, J. Alg. Geom. \textbf{11} (2002), 363--392.

\bibitem{Hart} R.~Hartshorne, \textit{Residues and duality}, Lecture notes of a seminar on the work of A. Grothendieck, given at Harvard 1963/64, With an appendix by P. Deligne, Lecture Notes in Mathematics, No.20 Springer-Verlag, Berlin-New York 1966. 

\bibitem{Hi} M.~Hirokado, Canonical singularities of dimension three in characteristic 2 which do not follow Reid's rules, arXiv:1607.08664, preprint (2016). 

\bibitem{HIS} M.~Hirokado, H.~Ito and N.~Saito, Three dimensional canonical singularities in codimension two in positive characteristic, J.~Algebra \textbf{373} (2013), 207--222. 

\bibitem{HH} M.~Hochster and C.~Huneke, Tight closure and strong $F$-regularity, M\'em. Soc. Math. France \textbf{38} (1989), 119--133. 

\bibitem{IR} S.~Ishii and J.~Reguera, Singularities in arbitrary characteristic via jet schemes, arXiv:1510.05210, preprint (2015). 

\bibitem{Ko} J.~Koll\'ar, \textit{Singularities of the minimal model program}, Cambridge Tracts in Mathematics, 200, Cambridge University Press, Cambridge, 2013. 

\bibitem{Ku} E.~Kunz, On Noetherian rings of characteristic $p$, Amer. J. Math. \textbf{98} (1976), 999--1013. 

\bibitem{Li} J.~Lipman, Rational singularities, with applications to algebraic surfaces and unique factorization, Inst. \'{H}autes Etudes Sci. Publ. Math., No. 36, 1969, 195--279.

\bibitem{MP} J.~McKernan and Y.~Prokhorov, Threefold thresholds, Manuscripta Math. \textbf{114} (2004), 281--304.

\bibitem{Re} M.~Reid, Canonical 3-folds, in: A. Beauville (Ed.), Journ\'{e}es de g\'{e}om\'{e}trie alg\'{e}brique d'Angers 1979, Sijthoff $\&$ Noordhoff, Alphen, 1980, pp. 273--310. 

\bibitem{SS} K.~Schwede and K.~E.~Smith, Globally $F$-regular and log Fano varieties, Adv.
Math. \textbf{224} (2010), 863--894.

\bibitem{ST} K.~Schwede and K.~Tcuker, On the behavior of test ideals under finite morphisms, J. Algebraic Geom. \textbf{23} (2014), 399--443.  

\bibitem{SZ} K.~Schwede and W.~Zhang, Bertini theorems for $F$-singularities, Proc. London Math. Soc. \textbf{107} (2013), 851--874. 

\bibitem{Ta} H.~Tanaka, Minimal model programme for excellent surfaces, arXiv:1608.07676, preprint (2016).

\bibitem{Wk} K.~Watanabe, Plurigenera of normal isolated singularities. I, Math. Ann. \textbf{250} (1980), 65--94. 

\bibitem{Wa} K.-i.~Watanabe, $F$-regular and $F$-pure normal graded rings, J. Pure Appl. Algebra  \textbf{71} (1991), 341--350.

\end{thebibliography}
\end{document}